\documentclass[10pt,twoside]{article}
\usepackage{indentfirst}
\usepackage{amsmath}
\usepackage{mathrsfs}
\usepackage{amsthm}
\usepackage{amssymb}
\usepackage{fancybox}
\usepackage{cmap}
\usepackage{url}
\usepackage{fancyhdr}
\usepackage{graphicx}
\usepackage{float}
\usepackage{color}
\graphicspath{{figures/}}
\newcommand{\be}{\begin{eqnarray}}
	\newcommand{\ee}{\end{eqnarray}}
\newcommand{\bestar}{\begin{eqnarray*}}
	\newcommand{\eestar}{\end{eqnarray*}}
\newtheorem{thm}{\textbf{Theorem}}
\newtheorem{lemma}{\textbf{Lemma}}

\newtheorem{corollary}{\textbf{Corollary}}
\usepackage{amsfonts}
\usepackage{fancyhdr}
\usepackage{titlesec}
\usepackage{cite}
\usepackage{ifthen}
\usepackage{pifont}
\usepackage{stmaryrd}
\usepackage{setspace}
\usepackage{indentfirst}
\usepackage{amsmath,amssymb,amscd,bbm,amsthm,mathrsfs,dsfont}
\input amssym.def

\titleformat{\section}{\centering\large\bfseries}{\S\arabic{section}}{1em}{}
\newboolean{first}
\setboolean{first}{true}

\begin{document}

\setlength\abovedisplayskip{2pt}
\setlength\abovedisplayshortskip{0pt}
\setlength\belowdisplayskip{2pt}
\setlength\belowdisplayshortskip{0pt}

\title{\bf \Large Open and increasing paths on $N$-ary trees with different fitness values\author{REN Tian-xiang~~~~~~WU Jin-wen}\date{}} \maketitle
 \footnote{Received: 2022-09-19.}
 \footnote{Digital Object Identifier(DOI): 10.1007/s11766-016.}
\begin{center}
\begin{minipage}{135mm}

{\bf \small Abstract}.\hskip 2mm {\small
Consider a rooted $N$-ary tree. For every vertex of this tree, we atttach an i.i.d. Bernoulli random variable. A path is called open if all the random variables that are assigned on the path are 1. We consider limiting  behaviors for the number of open paths from the root to leaves and the longest open path. In addition, when all fitness values are i.i.d. continuous random variables, some asymptotic properties of the longest increasing path are proved.}
\end{minipage}
\end{center}

\thispagestyle{fancyplain} \fancyhead{}
\fancyhead[L]{\textit{Appl. Math. J. Chinese Univ.}\\
2022, **} \fancyfoot{} \vskip 10mm

	\section{Introduction}
	
	Since Broadbent and Hammersley introduced the percolation models in 1957, it has become a cornerstone of probability theory and statistical physics, with applications ranging from molecular dynamics to star formation. 
	Standard percolation theory is concerned with the loss of global connectivity in a graph when vertices or bonds are randomly removed, as quantified by the probability for the existence of an infinite cluster of contiguous vertices. Here we consider a kind of percolation problem.

	Let $\mathbb{T}^{(N)}$ be a rooted $N$-ary tree in which each vertex has exactly $N$ children, of which the root we denote by $o$. Each vertex $\sigma \in \mathbb{T}^{(N)}$ is assigned a random variable  $X_{\sigma}$, called its fitness. The fitness values are independent and identically distributed (i.i.d.) random variables. We consider the Bernoulli random variable $X_{\sigma}$ with $\mathbb{P}(X_{\sigma}=0)=1-\mathbb{P}(X_{\sigma}=1)=1-p$. 
	We say that the path $P=\sigma_0\sigma_1\sigma_2\cdots \sigma_k$ with length $l(P)=k$ is  a path down the tree if $P$ starts at any vertex and descends into children until it stops at some node, i.e., $|\sigma_0|=|\sigma_1|-1=|\sigma_2|-2=\cdots=|\sigma_k|-k$, and $P$ is an open path if $X_{\sigma_0}=X_{\sigma_1}=X_{\sigma_2}=\cdots=X_{\sigma_k}=1$, where $|\sigma|$ is defined as the distance from $o$ to $\sigma$.

	Firstly, we care about the statistics of such open paths from the root to leaves, specifically the probability for the existence of open paths that span the entire tree. Let $\Theta_n$ be the number of open paths from the root to leaves. Since the probability that a given path from the root to the leaf is open is $p^{n+1}$,  by linearity the first moment of $\Theta_n$ follows immediately,
	\bestar
	\mathbb{E}(\Theta_n)=N^np^{n+1}.
	\eestar
	\begin{thm}\label{th1}
 For any $N\geq 2$, we have that
  \bestar
 \lim_{n\to\infty}\mathbb{P}(\Theta_n\ge 1)\left\{
 \begin{array}{ll}
   \ge(Np-1)/(N-1), & \mbox{if}~ p>1/N;\\
  =0, & \mbox{if}~ p\leq 1/N.
 \end{array}
 \right.
 \eestar
\end{thm}
This implies that there is a phase transition at $p=1/N$.

	Since there may not be an open path from the root to the leaves on $N$-ary tree, a natural idea is to study the longest length of open paths from any vertex. Let $\mathbb{T}_{n}^{(N)}$ be the subgraph of $\mathbb{T}^{(N)}$ induced by the set of nodes with levels not exceeding $n$, and $\mathcal{P}_{n}$ be the set of paths down the $\mathbb{T}_{n}^{(N)}$, then the length of the longest open path from any vertex can be defined as
	\bestar
	L_{N,n}:=\max_{P\in\mathcal{P}_{n}}\lbrace l(P):\; P ~is~an~open~path~down~the~tree\rbrace.
	\notag
	\eestar
	
	For $N=1$, $L_{1,n}$ is called the longest head run, which has important applications in biology, reliability theory, finance, and nonparametric statistics (see \cite{ref2, ref3,ref4,ref16}). One of the most important results on asymptotic behaviors of  $L_{1,n}$ is the following Erd\H{o}s-R\'enyi Law \cite{ref6, ref7}:
	\bestar
	\frac{L_{1,n}}{\ln{n}}\stackrel{a.s.~}{\longrightarrow}-\frac{1}{\ln{p}}.\label{1}
	\eestar
	Besides the law of large numbers, the possible asymptotic distribution of $L_{1,n}$ is also discussed (see \cite{ref1, ref8, ref9, ref10, ref11,ref12, ref14, ref15}). It is proved in \cite{ref10} that $L_{1,n}+\ln{n}/\ln{p}$ possesses no limit distribution. F\"oldes proved the distribution of the longest head run in the first $n$ trials of a coin tossing sequence which contains at most $T$ tails in \cite{ref8}, then extended it in \cite{ref9} and obtained that for $k\in\mathbb{Z}$,
	\bestar
	\mathbb{P}(L_{1,n}-[\ln{n}]<k)=\exp\lbrace{-2^{-(k+1-\lbrace{\ln{n}\rbrace})}\rbrace}+o(1),\label{2}
	\eestar
	where $\lbrace{\ln{n}\rbrace}=\ln{n}-[\ln{n}]$, $[\ln{n}]$ is the integer part of $\ln{n}$. A more accurate result of the limit distribution of $L_{1,n}$ is obtained in \cite{ref15}. Then the large deviation theorem for $L_{1,n}$ is shown in \cite{ref12}, and an estimation of the accuracy of approximation in terms of the total variation distance is established in \cite{ref14} for the first time.
	
	For any $N\geq 2$, we obtain the law of large numbers for $L_{N,n}$.
	\begin{thm}\label{th2}
		For any $N\geq 2$, $0<p<1$, as $n\to\infty$,
		\be
		\frac{L_{N,n}}{n}\overset{p}{\longrightarrow} 
		\begin{cases}
			-\log_p{N},&\text{if $0<p<1/N$};\\
			1,&\text{if $1/N\leq p<1$}.\\
		\end{cases}
		\notag
		\ee
	\end{thm}
	The rest of this paper is organised as follows. Sections 2-3 are devoted to the proof of our main results
	stated above. In section 4, we introduce
	the length of the longest increasing path when the fitness values are i.i.d. continuous random variables, which allows to give some asymptotic properties of the length of the path. Finallly, related results are proved in section 5.
\section{Proof of Theorem \ref{th1}}
\begin{proof}
We define indicator variables $\xi_i$ for each path $i\in \{1,2,\dots,N^n\}$ by whether the $i$-th path is open or not.
The $\xi_i$ are dependent random variables but identically distributed with
\bestar
\mathbb{E}(\xi_i)=\mathbb{E}(\xi_i^2)=\mathbb{P}(\xi_i=1)=p^{n+1}.
\eestar
This yields $\Theta_n=\sum\limits_{i=1}^{N^n}\xi_i$ and hence
\bestar
\mathbb{E}(\Theta_n^2)&=&\sum_{i=1}^{N^n}\mathbb{E}(\xi_i^2)+\sum_{i,j=1;i\neq j}^{N^n}\mathbb{E}(\xi_i\xi_j)=\sum_{i=1}^{N^n}\mathbb{E}(\xi_i)+\sum_{i,j=1;i\neq j}^{N^n}\mathbb{E}(\xi_i\xi_j)\\
&=&\mathbb{E}(\Theta_n)+\sum_{i,j=1;i\neq j}^{N^n}\mathbb{P}(\xi_i=1,\xi_j=1)\\
&=&\mathbb{E}(\Theta_n)+\sum_{k=1}^n(N-1)N^{n+k-1}p^{n+k+1}.
\eestar
Therefore 
\be
\mathbb{E}(\Theta_n^2)=\left\{
\begin{array}{ll}
 p+np(1-N^{-1}), & \mbox{if}~ p=1/N;\\
 \mathbb{E}(\Theta_n)+\frac{N-1}{Np}(\mathbb{E}(\Theta_n))^2+\frac{N-1}{N}\cdot\frac{N^{n+1}p^{n+2}(1-N^{n-1}p^{n-1})}{1-Np}, & \mbox{if}~ p\neq 1/N.\label{3}
\end{array}
\right.
\ee

We denote by $Q_n=1-\mathbb{P}(\Theta_n\ge 1)$ the probability that there is no path in the tree of level $n$. Since a tree of level $n+1$ can be thought of as a root which is connected to $N$-ary trees of level $n$, this leads to 
 \be
 Q_1=p(1-p)^N+1-p;\\
 Q_{n+1}=pQ_n^N+1-p.\label{dd}
 \ee
 Obviously $\{Q_n\}_{n=1,2,\dots}$ is an increasing sequence and has an upper bound of 1, it follows from the criteria of monotone bounded that this sequence has a limit. Then the $\mathbb{P}(\Theta_n\ge 1)$ also has a limit.

Notice that $\Theta_n$ is a random variable which takes only integer values and has a finite second moment. By (\ref{3}), we have the following inequality
\be
\mathbb{E}(\Theta_n)\ge\mathbb{P}(\Theta_n\ge 1)\ge \frac{(\mathbb{E}(\Theta_n))^2}{\mathbb{E}(\Theta_n^2)}=\left\{
\begin{array}{ll}
 \frac{p}{1+n(1-N^{-1})}, & \mbox{if}~ p=1/N;\\
 \frac{\mathbb{E}(\Theta_n)}{1+(N-1)\mathbb{E}(\Theta_n)/Np+S(N,p)}, & \mbox{if}~ p\neq 1/N.\label{4}
\end{array}
\right.
\ee
where 
\bestar
S(N,p)=\frac{N-1}{N}\cdot\frac{Np-N^np^n}{1-Np}.
\eestar
As $n\to\infty$, we have
 \bestar
\frac{N}{N-1}\cdot\frac{S(N,p)}{\mathbb{E}(\Theta_n)}=\frac{Np-(Np)^n}{(Np)^n(1-Np)p}=\frac{(Np)^{1-n}-1}{(1-Np)p}\left\{
 \begin{array}{ll}
 \le \frac{1}{(Np-1)p}, & \mbox{if}~ p>1/N;\\
 \to \infty, & \mbox{if}~ p<1/N.
 \end{array}
 \right.
 \eestar
 By combining this with (\ref{4}), we get
 \bestar
 \lim_{n\to\infty}\mathbb{P}(\Theta_n\ge 1)\left\{
 \begin{array}{ll}
  \ge(Np-1)/(N-1), & \mbox{if}~ p>1/N;\\
  =0, & \mbox{if}~ p<1/N.
 \end{array}
 \right.
 \eestar
For the case $p=1/N$, we define $Q=\lim\limits_{n\to\infty}Q_n$ and let
 \bestar
 f(x)=px^N-x+1-p, ~x\in[0,1].
 \eestar
 Note that $f'(x)=x^{N-1}-1\leq 0$ for $x\in[0,1]$, so $f(x)$ is monotonically nonincreasing. Furthermore,
 \bestar
 f(0)=1-p>0, f(1)=0.
 \eestar
 By applying (\ref{dd}), this implies $Q=1$. Therefore we complete the proof of Theorem \ref{th1}.
\end{proof}
	
	\section{Proof of Thereom \ref{th2}}
	In this section, $N\ge 2$ is a fixed positive integer.
	Let $\mathcal{P}_{n,k}$ be the set of   paths down  $\mathbb{T}^{(N)}_n$ with length $k$, it's clear that
	\be
	M:=\#\mathcal{P}_{n,k}=\sum_{j=0}^{n-k}N^{k+j}=\frac{N^{n+1}-N^{k}}{N-1}.\label{ee}
	\ee
	Define $T_{n,k}$ to be the number of open paths in $\mathcal{P}_{n,k}$:
	\bestar
	T_{n,k}=\sum_{P\in \mathcal{P}_{n,k}} \mathbb{I}\, (P ~\mbox{is  an open path down the tree}).
	\eestar
	It's clear that
	\be
	\mathbb{E}(T_{n,k})=\sum_{P\in\mathcal{P}_{n,k}}\mathbb{P}(P~is~an~open~path~down~the~tree)=Mp^{k+1}=\frac{N^{n+1}-N^{k}}{N-1}p^{k+1}.\notag
	\ee
	\begin{lemma}\label{le1}
		For any path $P\in\mathcal{P}_{n,k}$, we denote by $V(P)$ be the vertex set of $P$, and define
		\be
		a_m:=\#\lbrace (P,~\tilde{P}):~\#(V(P)\cap V(\tilde{P}))=m,~P,~ \tilde{P} \in\mathcal{P}_{n,k}\rbrace,
		\notag
		\ee
		then we have
		\be
		a_{m}\leq 2MN^{k-m+1},\quad\quad m=1,2,\cdots,k.
		\notag
		\ee
	\end{lemma}
	\begin{proof}
		Given a path $P=\sigma_0\sigma_1\sigma_2\cdots\sigma_k\in\mathcal{P}_{n,k}$ with $|\sigma_0|=j$, let $B(m,j)$ denote the number of the path $\tilde{P}\in\mathcal{P}_{n,k}$ which intersects $P$ $m(1\le m\le k+1)$ vertices, then 
		\be
		a_m=\sum_{j=0}^{n-k}N^{k+j}B(m,j).\label{5}
		\ee
		For $m=k+1$, it's clear that $B(m,j)=1$ and $a_{k+1}=M$.
		For $1\leq m\leq k$, when $k-m+1\leq j\leq n-2k+m-1$, we have
		\bestar
		B(m,j)&=&(k-m+1)(N-1)N^{k-m}+N^{k-m+1}+\sum_{i=1}^{k-m}(N-1)N^{k-m-i}+1\\
		&=&(k-m+1)(N-1)N^{k-m}+N^{k-m+1}+N^{k-m}\\
		&=&(k-m)(N-1)N^{k-m}+2N^{k-m+1};
		\eestar
		when $0\leq j\leq k-m$,
		\bestar
		B(m,j)&=&(k-m+1)(N-1)N^{k-m}+N^{k-m+1}+\sum_{i=1}^{j}(N-1)N^{k-m-i}\\
		&=&(k-m+1)(N-1)N^{k-m}+N^{k-m+1}+N^{k-m}-N^{k-m-j}\\
		&=&(k-m)(N-1)N^{k-m}+2N^{k-m+1}-N^{k-m-j};
		\eestar
		similarly, when $n-2k+m\leq j\leq n-k$,
		\bestar
		B(m,j)=(n-k-j)(N-1)N^{k-m}+N^{k-m+1}.
		\eestar
		Those together with (\ref{5}), imply that 
		\be
		\begin{aligned}
			a_{m}& =\sum_{j=0}^{k-m}N^{k+j}B(m,j)+\sum_{j=k-m+1}^{n-2k+m-1}N^{k+j}B(m,j)+\sum_{j=n-2k+m}^{n-k}N^{k+j}B(m,j)\\
			& =\sum_{j=0}^{n-k}N^{2k-m+j+1}+\sum_{j=0}^{n-2k+m-1}\Big[(k-m)(N-1)N^{2k-m+j}+N^{2k-m+1+j}\Big]\\
			& \quad\quad+\sum_{j=n-2k+m}^{n-k}(n-k-j)(N-1)N^{2k-m+j}-\sum_{j=0}^{k-m}N^{2k-m}\\
			& =\frac{N^{n+k-m+2}-N^{2k-m+1}}{N-1}+(k-m)(N^n-N^{2k-m})+\frac{N^{n+1}-N^{2k-m+1}}{N-1}\\
			& \quad\quad+\sum_{j=1}^{k-m}(N-1)jN^{n+k-m-j}-(k-m+1)N^{2k-m}\\
			& =\frac{N^{n+k-m+2}-N^{2k-m+1}}{N-1}+\frac{N^{n+k-m+1}-N^{2k-m+1}}{N-1}-(2k-2m+1)N^{2k-m}\\
			&\leq 2MN^{k+1-m},
		\end{aligned}
		\notag
		\ee
		where we have used the formula $\sum\limits_{i=1}^niq^i=\frac{q(1-q^n)}{(1-q)^2}-\frac{nq^{n+1}}{1-q}.$
		Thus completing the proof of Lemma \ref{le1}.
	\end{proof}
	\begin{lemma}\label{le2}
		For any $N\geq 2$,
		\be
		Var(T_{n,k})\leq
		\begin{cases}
			2(1-Np)^{-1}\mathbb{E}(T_{n,k}),&\text{if $0<p<1/N$};\\
			2(k+1)\mathbb{E}(T_{n,k}),&\text{if $p=1/N$};\\
			2N^{k+1}p^{k+1}(Np-1)^{-1}\mathbb{E}(T_{n,k}),&\text{if $1/N<p<1$}.\\
		\end{cases}
		\notag
		\ee
	\end{lemma}
	\begin{proof}
		We say that two paths $P$ and $\tilde{P}$ are vertex-disjoint if $V(P)\cap V(\tilde{P})=\emptyset$. Note that $\mathbb{I}(P~is~an~open~path~down~the~tree)$ and $\mathbb{I}(\tilde{P}~is~an~open~path~down~the~tree)$ are independent if $P$ and $\tilde{P}$ are vertex-disjoint. Then it follows from Lemma \ref{le1} that
		\be
		\begin{aligned}
			Var(T_{n,k})&=\sum_{P,\tilde{P}\in\mathcal{P}_{n,k}}(\mathbb{P}(P~and~\tilde{P}~are~open~paths)-\mathbb{P}(P~is~an~open~path)^2) \\
			&\leq\sum_{V(P)\cap V(\tilde{P})\neq\emptyset\atop P,\;\tilde{P}\in\mathcal{P}_{n,k}}\mathbb{P}(P~and~\tilde{P}~are~open~paths)\\
			&=\sum_{m=1}^{k+1}\sum_{|V(P)\cap V(\tilde{P})|=m\atop P,\;\tilde{P}\in\mathcal{P}_{n,k}}\mathbb{P}(P~and~\tilde{P}~are~open~paths)\\
			&=\sum_{m=1}^{k+1}a_mp^{2k+2-m}\leq\sum_{m=1}^{k+1}2MN^{k-m+1}p^{2k+2-m},
		\end{aligned}
		\notag
		\ee
		which yields Lemma \ref{le2}.
	\end{proof}
	
	\begin{proof}[Proof of Theorem \ref{th2}]
		When $0<p<1/N$, for any  sequence $k_n\to\infty$, 
		\be
		\begin{aligned}
			\mathbb{E}(T_{n,k_n})&=\frac{N^{n+1}-N^{k_n}}{N-1}p^{k_n+1}=\exp{\Big\{ k_n\log{p}+n\log{N}+O(1)\Big\}}.\label{6}
		\end{aligned}
		\ee
		For any $\epsilon>0$, we take $k_n=[(-\frac{\log{N}}{\log{p}}+\epsilon)n]$ and $\tilde{k}_n=[(-\frac{\log{N}}{\log{p}}-\epsilon)n]$. Then, by applying (\ref{6}), we have
		\be
		\mathbb{P}(L_{N,n}\geq k_n)=\mathbb{P}(T_{n,k_n}\geq 1)\leq\mathbb{E}(T_{n,k_n})\to 0,\quad n\to\infty,\label{7}
		\ee
		and furthermore, by Lemma \ref{le2} and Chebyshev's inequality,
		\be
		\begin{aligned}
			\mathbb{P}(L_{N,n}<\tilde{k}_n)&=\mathbb{P}(T_{n,\tilde{k}_n}=0)\leq\mathbb{P}(|T_{n,\tilde{k}_n}-\mathbb{E}(T_{n,\tilde{k}_n})|\geq\mathbb{E}(T_{n,\tilde{k}_n}))\\
			&\leq\frac{Var(T_{n,\tilde{k}_n})}{(\mathbb{E}(T_{n,\tilde{k}_n}))^2}
			\leq\frac{2}{(1-Np)\mathbb{E}(T_{n,\tilde{k}_n})}\to 0,\quad n\to\infty.\label{8}
		\end{aligned}
		\ee
		When $p\ge1/N$, let $c_n=[(1-\epsilon)n]$, by same method, as $n\to\infty$, we have
		\be
			\begin{aligned}
		\mathbb{P}(L_{N,n}< c_n)&=\mathbb{P}(T_{n,c_n}=0)\leq\frac{Var(T_{n,c_n})}{(\mathbb{E}(T_{n,c_n}))^2}\\
	&	\le\left\{
		\begin{array}{ll}
		\frac{2(c_n+1)}{\mathbb{E}(T_{n,c_n})}\to 0, & \mbox{if}~ p=1/N;\\
		\frac{2N^{c_n+1}p^{c_n+1}}{(Np-1)\mathbb{E}(T_{n,c_n})}\to 0, & \mbox{if}~ p >1/N.\label{9}
		\end{array}
		\right.
			\end{aligned}
		\ee
		Hence, from (\ref{7}), (\ref{8}) and (\ref{9}), we can get Theorem 1.
	\end{proof}
	
	\section{The longest increasing path}
	When the fitness values are i.i.d. continuous random variables, we say that the path $P=\sigma_0\sigma_1\cdots\sigma_k$ with length $l(P)=k$ is an increasing path if $X_{\sigma0}<X_{\sigma1}<\cdots<X_{\sigma k}$.
	Nowak and Krug called this accessibility percolation and derived an asymptotically exact expression for the probability that there is at least one accessible path from the root to the leaves in \cite{ref17}. In the evolutionary biology literature, these increasing paths are known as selectively accessible. The probability that there exists an incresaing path from the root to the leaves is discussed in \cite{ref17,ref19}.

Besides, the number of increasing paths on other graphs has been studied extensively. It has been considered on the $N$-dimensional binary hypercube $\{0,1\}^N$ in \cite{ref20,ref21,ref22}. On infinite spherically symmetric trees, the increasing path has also been studied in \cite{ref23}.

	Since we only care about whether the fitness values along the path are in increasing order, as long as the random variables are continuous, changing the precise distribution will not influence the results. Without loss of generality, we assume that all the random variables are mutually independent and uniformly distributed on $[0,1]$, then we can consider the length of the longest increasing path, which can be defined as
	\be
	\tilde{L}_{N,n}=\max\lbrace l(P):\; P~is~an~increasing~path~down~the~tree\rbrace.
	\notag
	\ee
	
	It is studied in \cite{ref5} that the limit distribution of $\tilde{L}_{N,n}$ for $N=1$, and they obtained that
	\be
	\lim_{n\to\infty}\mathbb{P}([f_n]-1\leq \tilde{L}_{1,n}\leq[f_n]+1)=1,\label{11}
	\ee
	where $f_n=\frac{\log{n}}{b_n}-\frac{1}{2}$ and $b_n$ is the solution of the equation $b_ne^{b_n}=e^{-1}\log {n}$. Our result is as follows:
	\begin{thm}\label{th3}
		For any $N\geq 2$, 
		\be
		\lim_{n\to\infty}\mathbb{P}([f_{N,n}]-1\leq \tilde{L}_{N,n}\leq[f_{N,n}]+1)=1,
		\notag
		\ee
		where $f_{N,n}=\frac{n\log {N}}{b_{N,n}}-\frac{1}{2}$ and $b_{N,n}$ is the solution of the equation $b_{N,n}e^{b_{N,n}}=e^{-1}n\log {N}$.
	\end{thm}
	\begin{corollary}\label{co1}
		Let $f_{N,n}$ be defined in Theorem \ref{th3} and $\lbrace f_{N,n}\rbrace=f_{N,n}-[f_{N,n}]$ be the fractional part of $f_{N,n}$. Assume that $(n_k,k\geq 1)$ and $(n'_k,k\geq 1)$ are subsequences satisfying that
		$$\limsup_{k\to\infty}\lbrace f_{N,n_k}\rbrace<1,\quad\quad\liminf_{k\to\infty}\lbrace f_{N,n'_k}\rbrace>0.$$
		Then
		\be
		\begin{aligned}
			&\lim_{k\to\infty}\mathbb{P}([f_{N,n_k}]-1\leq \tilde{L}_{N,n_k}\leq[f_{N,n_k}])=1;\label{12}
		\end{aligned}
		\ee
		\be
		\begin{aligned}
			&\lim_{k\to\infty}\mathbb{P}([f_{N,n'_k}]\leq \tilde{L}_{N,n'_k}\leq[f_{N,n'_k}]+1)=1.\label{13}
		\end{aligned}
		\ee
		Furthermore, if $0<\liminf_{k\to\infty}\lbrace f_{N,n''_k}\rbrace\leq \limsup_{k\to\infty}\lbrace f_{N,n''_k}\rbrace<1$ holds for the subsequence $(n''_k,k\geq 1)$, then
		\be
		&\lim_{k\to\infty}\mathbb{P}(\tilde{L}_{N,n''_k}=[f_{N,n''_k})=1.\label{14}
		\ee
	\end{corollary}
	\begin{corollary}\label{co2}
		Let $f_{N,n}$ and $\lbrace f_{N,n}\rbrace$ be defined in Corollary 1. Assume that $(n_k,k\geq 1)$ and $(n'_k,k\geq 1)$ are subsequences satisfying that
		\be
		\lbrace f_{N,n_k}\rbrace\log{f_{N,n_k}}\to a\in [0,\infty]\quad\quad and \quad\quad\lbrace f_{N,n_k}\rbrace\to 0,\notag
		\ee
		and
		\be
		(1-\lbrace f_{N,n'_k}\rbrace)\log{f_{N,n'_k}}\to a\in [0,\infty]\quad\quad and \quad\quad\lbrace f_{N,n'_k}\rbrace\to 1,\notag
		\ee
		Then
		\be
		\begin{aligned}
			\mathbb{P}(\tilde{L}_{N,n_k}=[f_{N,n_k}-1)=1-\mathbb{P}(\tilde{L}_{N,n_k}=[f_{N,n_k}])=\exp{\Big\{-\frac{Ne^a}{\sqrt{2\pi}(N-1)}\Big\}};\\
			\mathbb{P}(\tilde{L}_{N,n'_k}=[f_{N,n'_k}])=1-\mathbb{P}(\tilde{L}_{N,n'_k}=[f_{N,n'_k}]+1)=\exp{\Big\{-\frac{Ne^{-a}}{\sqrt{2\pi}(N-1)}\Big\}}.\notag
		\end{aligned}
		\ee
	\end{corollary}
	
	\section{Proof of Theorem \ref{th3} and corollary \ref{co1} and \ref{co2}}
	\begin{lemma}\label{le3}
		Let $X_1,\cdots,X_n$ indicator variables with $\mathbb{P}(X_i=1)=p_i,\;W=\sum_{i=1}^nX_i,$ and $\lambda=\mathbb{E}W=\sum_ip_i.$ For each $i$, let $N_i\subseteq\lbrace 1,\cdots,n\rbrace$ such that $X_i$ is independent of $\lbrace X_j:j\notin N_i\rbrace$. If $p_{ij}:=\mathbb{E}[X_iX_j]$ and $Z\sim Po$($\lambda$), then 
		\be
		d_{TV}(W,Z)\leq min\lbrace1,\lambda^{-1}\rbrace\left(\sum_{i=1}^n\sum_{j\in N_i}p_ip_j+\sum_{i=1}^n\sum_{j\in N_i/{i}}p_{ij}\right).
		\notag
		\ee
	\end{lemma}
	\begin{proof}
		See Theorem 4.7 in \cite{ref13}.
	\end{proof}
	\begin{lemma}\label{le4}Let $x \in \mathbb{R}$ and $f_{N,n}=\frac{n\log {N}}{b_{N,n}}-\frac{1}{2}$, where $b_{N,n}$ is the solution of the equation $b_{N,n} \mathrm{e}^{b_{N,n}}=\mathrm{e}^{-1} n\log N$. Then
		\be
		\lim_{n\to+\infty}\frac{N^n}{\Gamma(f_{N,n}+x+1)}=\lim_{n\to+\infty}\frac{1}{\sqrt{2\pi}}e^{-x\log{(f_{N,n}+x)}}=
		\begin{cases}
			+\infty,&\text{$x<0$},\\
			0,&\text{$x>0$},
		\end{cases}
		\notag
		\ee
		where $\Gamma(y)=\int_{0}^{\infty} \mathrm{e}^{-t} t^{y-1} \mathrm{~d} t$ denotes the Gamma function at the point $y$.
	\end{lemma}
	\begin{proof}
		Stirling's formula shows
		\be
		\Gamma(f_{N,n}+x+1)\sim\frac{\sqrt{2 \pi}(f_{N,n}+x)^{f_{N,n}+x+1 / 2}}{\mathrm{e}^{f_{N,n}+x}}.\notag
		\ee
		Hence, it suffices to show that
		\be
		\begin{aligned}
			\lim_{n\to\infty}\frac{N^n \mathrm{e}^{f_{N,n}+x}}{\sqrt{2 \pi}(f_{N,n}+x)^{f_{N,n}+x+1 / 2}} =
			\begin{cases}
				+\infty,&\text{$x<0$},\\
				0,&\text{$x>0$}.
			\end{cases}
			\notag
		\end{aligned}
		\ee
		We have 
		\be
		\begin{aligned}
			\frac{N^n \mathrm{e}^{f_{N,n}+x}}{\sqrt{2 \pi}(f_{N,n}+x)^{f_{N,n}+x+1 / 2}} &=\frac{N^n \exp \left(-f_{N,n}\lbrace\log (f_{N,n}+x)-1\rbrace\right)}{\sqrt{2 \pi(f_{N,n}+x)}} \mathrm{e}^{-x\lbrace\log (f_{N,n}+x)-1\rbrace}.
			\notag
		\end{aligned}
		\ee
		By the proof of Lemma 1 in \cite{ref5}, we have 
		\be
		\lim_{n\to\infty}\frac{N^n \exp \left(-f_{N,n}\lbrace\log (f_{N,n}+x)-1\rbrace\right)}{\sqrt{2 \pi(f_{N,n}+x)}}=\frac{e^{-x}}{\sqrt{2\pi}}.\notag
		\ee
		Since $\log (f_{N,n}+x)\to+\infty$, as $n\to\infty$, we obtain
		\be
		\lim_{n\to\infty}\frac{N^n \mathrm{e}^{f_{N,n}+x}}{\sqrt{2 \pi}(f_{N,n}+x)^{f_{N,n}+x+1 / 2}}=
		\begin{cases}
			+\infty,&\text{$x<0$},\\
			0,&\text{$x>0$}.
		\end{cases}
		\notag
		\ee
	\end{proof}
	Similarly with the proof of Theorem \ref{th2}, we define $\tilde{T}_{n,k}$ to be the number of increasing paths in $\mathcal{P}_{n,k}$:
	\bestar
	\tilde{T}_{n,k}=\sum_{P\in \mathcal{P}_{n,k}} \mathbb{I}\, (P ~\mbox{is  increasing}).
	\eestar
	Let $M$ be defined in (\ref{ee}), it's clear that
	\be
	\mathbb{E}(\tilde{T}_{n,k})=\sum_{P\in\mathcal{P}_{n,k}}\mathbb{P}(P~is~increasing)=\frac{M}{(k+1)!}.\label{ff}
	\ee
	\begin{lemma}\label{le5}
		Let $U$ be the Poisson distribution with mean $\lambda=\mathbb{E}(\tilde{T}_{n,k})$, then for any $N\geq 2$ and $q=8N/(2k+3)\in(0,1)$, we have
		\be
		d_{TV}(\tilde{T}_{n,k},U)\leq \frac{k+2}{(k+1)!}N^{k}+\frac{4N}{(k+2)(1-q)}.
		\notag
		\ee
	\end{lemma}
	\begin{proof}
		For each path $P\in\mathcal{P}_{n,k}$, define $N_P=\lbrace P':P'\in\mathcal{P}_{n,k},~P'~is~dependent~of~P\rbrace$. Then by applying Lemma \ref{le3}, we have
		\be
		\begin{aligned}
			d_{TV}(\tilde{T}_{n,k},Po(\lambda))&\leq\min\lbrace 1,\lambda^{-1}\rbrace(H_1+H_2),\label{15}
		\end{aligned}
		\ee
		where
		\bestar
		H_1=\sum_{P\in\mathcal{P}_{n,k}}\sum_{P'\in N_P}\frac{1}{((k+1)!)^2},~~H_2=\sum_{P\in\mathcal{P}_{n,k}}\sum_{P'\in N_P/\lbrace P\rbrace}\mathbb{P}(P ~and ~\tilde{P}~are ~increasing).
		\eestar
		
		For $H_1$, for $1\leq m\leq k$, recalling that by the proof of Lemma \ref{le1}, 
		\bestar
		\max_{j}B(m,j)=(k+1-m)(N-1)N^{k-m}+N^{k-m+1}+N^{k-m},
		\eestar
		which together with $B(k+1,j)=1$, implies that
		\be
		\begin{aligned}
			\max_{P\in\mathcal{P}_{n,k}}|N_P|&=\max_{j}\sum_{m=1}^{k+1}B(m,j)\leq\sum_{m=1}^{k+1}\max_{j}B(m,j)\\
			&\leq \sum_{m=1}^{k}\Big((k-m)(N-1)N^{k-m}+2N^{k-m+1}\Big)+1\\
			&=kN^{k}+\sum_{i=1}^{k-1}N^i\leq (k+2)N^{k}.\notag
		\end{aligned}
		\ee
		Together with (\ref{ff}), yields
		\be
		H_1\le\frac{M}{((k+1)!)^2}\max_{P\in\mathcal{\tilde{P}}_{n,k}}|N_P|\le\frac{(k+2)\lambda N^{k}}{(k+1)!}.\label{16}
		\ee
		
		To bound $H_2$, notice that for any two paths $P=x_0\cdots x_k,\tilde{P}=\tilde{x}_0\cdots\tilde{x}_k\in\mathcal{P}_{n,k}$ with $|x_0|\leq|\tilde{x}_0|$, if $P$ and $\tilde{P}$ are not vertex-disjoint, then there exsit integers $s,t$ such that $0\leq s\leq t\leq k$ and $x_{t-s+i}=\tilde{x}_i$ iff $0\leq i\leq s$. One gets
		\be
		\begin{aligned}
			&\mathbb{P}(P ~and ~\tilde{P}~are ~increasing~paths~down~the~tree)\\
			&~~~~=\int_0^1\frac{(1-x)^{2k-s+1}(2k+2-s-t)!}{(2k-s+1)!(k+1-s)!(k+1-t)!}dx\\
			&~~~~=\frac{(2k+2-s-t)!}{(2k+2-s)!(k+1-s)!(k+1-t)!}.
		\end{aligned}
		\notag
		\ee
		Since $\frac{(2k+2-s-t)!}{(2k+2-s)!(k+1-s)!(k+1-t)!}(0\leq s\leq k)$ is nonincreasing with respect to $t$, we obtain that
		\be
		 \frac{(2k+2-s-t)!}{(2k+2-s)!(k+1-s)!(k+1-t)!}\leq \frac{(2k+2-2s)!}{(2k+2-s)!(k+1-s)!(k+1-s)!}\label{gg}
		\ee
		holds for every $t\in[s,k]$.
	Thus it follows from Lemma \ref{le1} and (\ref{gg}) that
		\be
		\begin{aligned}
			H_2&=\sum_{P\in\mathcal{P}_{n,k}}\sum_{P'\in N_P/\lbrace P\rbrace}\mathbb{P}(P ~and ~\tilde{P}~are ~increasing)\\
			& \leq \sum_{s=1}^{k}a_s\frac{(2k+2-2s)!}{(2k+2-s)!(k+1-s)!(k+1-s)!}\\
			& =2\lambda\sum_{s=1}^{k}\frac{(2k+2-2s)!(k+1)!N^{k+1-s}}{(2k+2-s)!(k+1-s)!(k+1-s)!}\\
			& :=2\lambda\sum_{s=1}^{k}b_s.\\
		\end{aligned}
		\notag
		\ee
		Notice that
		$$\frac{b_{s-1}}{b_s}=\frac{(4k-4s+6)N}{(k-s+2)(2k-s+3)},$$
		which gets the maximum value of when $s=\frac{2k+3-\sqrt{2k+3}}{2}$, we obtain that
		\be
		\frac{b_{s-1}}{b_s}\leq\frac{8\sqrt{2k+3}N}{(1+\sqrt{2k+3})(2k+3)+\sqrt{2k+3})}\leq q.
		\notag
		\ee
		Then
		\bestar
		\begin{aligned}
			H_2\leq2\lambda\sum_{i=1}^{k-1}q^ib_{k}\leq\frac{4N\lambda}{(k+2)(1-q)},
		\end{aligned}
		\eestar
			which together with (\ref{15}) and (\ref{16}), completes the proof.
	\end{proof}

	\begin{proof}[Proof of Theorem \ref{th3}.]
		For simplicity, we write $D(N,k)=(k+2)N^k/(k+1)!+4N/(k+2)(1-q)$. It is obvious that
		\be
		\mathbb{P}(\tilde{L}_{N,n}<k)=\mathbb{P}(\tilde{T}_{n,k}=0).
		\notag
		\ee
		Thus, using Lemma \ref{le5}, we can derive upper and lower bounds for the distribution of $\tilde{L}_{N,n}$:
		\be
		e^{-\lambda}-D(N,k)\leq \mathbb{P}(\tilde{L}_{N,n}<k)\leq e^{-\lambda}+D(N,k).\label{18}
		\ee
		By the definition of $f_{N,n}$ and $b_{N,n}$, we have
		\be
		\mathbb{P}([f_{N,n}]-1\leq \tilde{L}_{N,n}\leq[f_{N,n}]+1)=\mathbb{P}(\tilde{L}_{N,n}<[f_{N,n}]+2)-\mathbb{P}(\tilde{L}_{N,n}< [f_{N,n}]-1),
		\notag
		\ee
		for $k=[f_{N,n}]+2$ and $k=[f_{N,n}]-1$, the condition of Lemma \ref{le5} is satisfied. Then from (\ref{18}), we get
		\be
		\left|\mathbb{P}(\tilde{L}_{N,n}<[f_{N,n}]+2)-\exp\left(-\frac{N^{n+1}-N^{[f_{N,n}]+1}}{(N-1)([f_{N,n}]+2)!}\right)\right|\leq D(N,[f_{N,n}]+2);\notag\\
		\left|\mathbb{P}(\tilde{L}_{N,n}<[f_{N,n}]-1)-\exp\left(-\frac{N^{n+1}-N^{[f_{N,n}]-2}}{(N-1)([f_{N,n}]-1)!}\right)\right|\leq D(N,[f_{N,n}]-1).\label{19}
		\ee
		Since $f_{N,n}\to+\infty$, the bounds $D(N,[f_{N,n}]+2)$ and $D(N,[f_{N,n}]-1)$ tend to 0 as $n\to\infty$. 
		By Lemma \ref{le4}, we can verify that
		\be
		\begin{aligned}
			&\lim_{n\to\infty}\frac{N^{n+1}-N^{[f_{N,n}]+1}}{(N-1)([f_{N,n}]+2)!}=\frac{N}{N-1}\lim_{n\to\infty}\frac{N^n}{([f_{N,n}]+2)!}=\frac{N}{N-1}\lim_{n\to\infty}\frac{N^n}{\Gamma([f_{N,n}]+3)}=0 \quad \text { and } \\
			&\lim_{n\to\infty}\frac{N^{n+1}-N^{[f_{N,n}]-2}}{(N-1)([f_{N,n}]-1)!}=\frac{N}{N-1}\lim_{n\to\infty}\frac{N^n}{([f_{N,n}]-1)!}=\frac{N}{N-1}\lim_{n\to\infty}\frac{N^n}{\Gamma([f_{N,n}])}=+\infty,\notag
		\end{aligned}
		\ee
		which, together with (\ref{19}), finally completes the proof of the theorem.
	\end{proof}
	
	\begin{proof}[Proof of Corollary \ref{co1}]
		We only prove (\ref{12}). The proof of the (\ref{13}) is along the same lines. (\ref{14}) is an immediate consequence of (\ref{12}) and (\ref{13}). Using (\ref{18}), we get 
		\be
		\begin{aligned}
			&\left|\mathbb{P}(\tilde{L}_{N,n_k}<[f_{N,n_k}]+1)-\exp\left(-\frac{N^{n_k+1}-N^{[f_{N,n_k}]}}{(N-1)([f_{N,n_k}]+1)!}\right)\right|\leq D(N,[f_{N,n_k}]+1).\label{20}
		\end{aligned}
		\ee
		Since $(n_k,k\geq 1)$ is a strictly increasing sequence such that $\limsup_{k\to\infty}\lbrace f_{N,n_k}\rbrace<1$, we have
		\be
		\begin{aligned}
			&\lim_{k\to\infty}\frac{N^{n_k+1}-N^{[f_{N,n_k}]}}{(N-1)([f_{N,n_k}]+1)!}=\frac{N}{N-1}\lim_{k\to\infty}\frac{N^{n_k}}{([f_{N,n_k}]+1)!}=\frac{N}{N-1}\lim_{k\to\infty}\frac{N^{n_k}}{\Gamma([f_{N,n_k}]+2)}=0\notag,
		\end{aligned}
		\ee
		which, combined with Theorem \ref{th3} and (\ref{20}), finally cmopletes the proof of the Corollary \ref{co1}.
	\end{proof}
	\begin{proof}[Proof of Corollary \ref{co2}]
		If $\lbrace f_{N,n_k}\rbrace\log{f_{N,n_k}}\to a\in [0,\infty]$, then 
		\be
		\begin{aligned}
			&\lim_{k\to\infty}\frac{N^{n_k+1}-N^{[f_{N,n_k}]-1}}{(N-1)([f_{N,n_k}])!}=\frac{N}{N-1}\lim_{k\to\infty}\frac{N^{n_k}}{([f_{N,n_k}])!}=\frac{N}{N-1}\lim_{k\to\infty}\frac{N^{n_k}}{\Gamma([f_{N,n_k}]+1)}=\frac{N}{\sqrt{2\pi}(N-1)}e^a\notag.
		\end{aligned}
		\ee
		
		From (\ref{18}), we have
		\be
		\begin{aligned}
			&\left|\mathbb{P}(\tilde{L}_{N,n_k}<[f_{N,n_k}])-\exp\left(-\frac{N^{n_k+1}-N^{[f_{N,n_k}]-1}}{(N-1)([f_{N,n_k}])!}\right)\right|\leq D(N,[f_{N,n_k}]),\notag
		\end{aligned}
		\ee
		which means 
		\be
		\begin{aligned}
			\mathbb{P}(\tilde{L}_{N,n_k}<[f_{N,n_k}])=\exp{\lbrace-\frac{Ne^a}{\sqrt{2\pi}(N-1)}\rbrace},\notag
		\end{aligned}
		\ee
		and hence (\ref{12}) implies that
		\be
		\begin{aligned}
			\mathbb{P}(\tilde{L}_{N,n_k}=[f_{N,n_k}]-1)=1-\mathbb{P}(\tilde{L}_{N,n_k}=[f_{N,n_k}])=\exp{\lbrace-\frac{Ne^a}{\sqrt{2\pi}(N-1)}\rbrace}.\notag
		\end{aligned}
		\ee
		
		If $(1-\lbrace f_{N,n'_k}\rbrace)\log{f_{N,n'_k}}\to a\in [0,\infty]$, then 
		\be
		\begin{aligned}
			&\lim_{k\to\infty}\frac{N^{n'_k+1}-N^{[f_{N,n'_k}]}}{(N-1)([f_{N,n'_k}]+1)!}=\frac{N}{N-1}\lim_{k\to\infty}\frac{N^{n'_k}}{([f_{N,n'_k}]+1)!}=\frac{N}{N-1}\lim_{k\to\infty}\frac{N^{n'_k}}{\Gamma([f_{N,n'_k}]+2)}=\frac{N}{\sqrt{2\pi}(N-1)}e^{-a}\notag,
		\end{aligned}
		\ee
		and from (\ref{20}), we have
		\be
		\begin{aligned}
			\mathbb{P}(\tilde{L}_{N,n'_k}<[f_{N,n'_k}]+1)=\exp{\Big\{-\frac{Ne^{-a}}{\sqrt{2\pi}(N-1)}\Big\}},\notag
		\end{aligned}
		\ee
		and hence (\ref{13}) implies that
		\be
		\begin{aligned}
			\mathbb{P}(\tilde{L}_{N,n'_k}=[f_{N,n'_k}])=1-\mathbb{P}(\tilde{L}_{N,n'_k}=[f_{N,n'_k}]+1)=\exp{\Big\{-\frac{Ne^{-a}}{\sqrt{2\pi}(N-1)}\Big\}}.\notag
		\end{aligned}
		\ee
		The proof of Corollary \ref{co2} is completed.
	\end{proof}

	\vskip 10mm

\noindent International Institute of Finance, School of Management, University of Science and Technology of China,  Hefei 230026, Anhui, China.\\
\indent Email:
rentianxiang7013@mail.ustc.edu.cn\\
\indent Email: wujw@mail.ustc.edu.cn

\begin{thebibliography}{} 
		
		\bibitem{ref1} Arratia, R., Goldstein, L. and Gordon, L.(1989). Two moments suffice for Poisson approximations: the Chen-Stein method. \textit{Ann. Probab.}, \textbf{17}, 9–25.
		
		\bibitem{ref2} Balakrishnan, N. and Koutras, M.V.(2001). Runs and scans with applications. \textit{Soc. Indus. Appl. Math.}, \textbf{45}, 137-139.
		
		\bibitem{ref3} Barbour, A.D., Holst, L. and Janson, S.(1992). Poisson approximation. \textit{The Clarendon Press, Oxford University Press, New York.}
		
		\bibitem{ref4} Bateman, G. (1948). On the power function of the longest run as a test for randomness in a sequence of alternatives. \textit{Biometrika}, \textbf{35}, 97–112.

		\bibitem{ref20}  Berestycki, J., Brunet, E. and Shi, Z.(2016). The number of accessible paths in the hypercube. \textit{Bernoulli}, \textbf{22}, 653–680.

		\bibitem{ref21}  Berestycki, J., Brunet, E. and Shi, Z.(2016). Accessibility percolation with backsteps. \textit{ALEA Lat. Am. J. Probab. Math. Stat.}, \textbf{14}, 45–62.
 		
		\bibitem{ref18} Broadbent, S.R. and Hammersley, J.M.(1957). Percolation processes. \uppercase\expandafter{\romannumeral1}. Crystals and mazes. \textit{Proc. Cambridge Philos. Soc.}, \textbf{53}, 629–641.
		
		\bibitem{ref5} Chryssaphinou, O. and Vaggelatou, E.(2001). Compound Poisson approximation for long increasing sequences. \textit{J. Appl. Probab.}, \textbf{38}, 449-463.

		\bibitem{ref23} Coletti, C. F., Gava, R. J. and Rodriguez, P. M.(2018). On the existence of accessibility in a tree-indexed percolation model. \textit{Physica A: Statistical Mechanics and its Applications}, \textbf{492}, 382-388.
		
		\bibitem{ref6} Erd\H{o}s, P. and R\'enyi, A.(1970). On a new law of large numbers. \textit{J. Analyse Math.}, \textbf{23}, 103–111.
		
		\bibitem{ref7} Erd\H{o}s, P. and R\'ev\'esz, P.(1975). On the length of the longest head-run. Topics in information theory. \textit{Colloq. Math. Soc. J\'anos Bolyai}, \textbf{16}, 219–228.
		
		\bibitem{ref8} F\"oldes, Ant\'onia(1975). On the limit distribution of the longest head run. \textit{Mat. Lapok}, \textbf{26}, 105-116.
		
		\bibitem{ref9} F\"oldes, Ant\'onia(1979). The limit distribution of the length of the longest head-run. \textit{Period. Math. Hungar.}, \textbf{10}, 301–310.
		
		\bibitem{ref10} Goncharov, V.L.(1962). On the field of combinatory analysis. \textit{Amer. Math. Soc. Transl.}, \textbf{19}, 1–46.
		
		\bibitem{ref11} Gordon, L., Schilling, M.F. and Waterman, M.S.(1986). An extreme value theory for long head runs. \textit{Probab. Theory. Relat. Fields}, \textbf{72}, 279–287.

		\bibitem{ref22} Hegarty, P. and Martinsson, A.(2014). On the existence of accessible paths in various models of fitness landscapes. \textit{The Annals of Applied Probability}, \textbf{24}, 1375–1395.


		\bibitem{ref12} Mao Y. and Wang F. and Wu X.(2015). Large deviation behavior for the longest head run in an IID Bernoulli sequence. \textit{J. Theoret. Probab.}, \textbf{28}, 259-268.
		
		\bibitem{ref13} Nathan, R.(2011). Fundamentals of Stein’s method. \textit{Probab. Surv.}, \textbf{8}, 210–293.
		
		\bibitem{ref17} Nowak, S. and Krug, J.(2013). Accessibility percolation on $n$-trees. \textit{Europhysics Letters}, \textbf{101}, 66004.
		
		\bibitem{ref14} Novak, S. Y.(2017). On the length of the longest head run. \textit{Statist. Probab.~Lett.}, \textbf{130}, 111-114.
		
		\bibitem{ref15} R\'ev\'esz, P.(1990). Regularities and irregularities in a random 0, 1 sequence. \textit{Statist. Hefte}, \textbf{31}, 95-101.

		\bibitem{ref19} Roberts, M.I. and Zhao, L.Z.(2013). Increasing paths in regular trees. \textit{Electronic Communications in Probability}, \textbf{18}, Paper No.87.

		
		\bibitem{ref16} Schwager S.J.(1983). Run probabilities in sequences of Markov-dependent trials. \textit{J. Amer. Statist. Assoc.}, \textbf{78}, 168–175.
		
	
		
	\end{thebibliography}
\end{document}